\documentclass{amsart}
\usepackage{amssymb}
\usepackage{graphicx}
\usepackage{amsmath,amsthm,amscd}
\pdfoutput=1 

\newtheorem{proposition}{Proposition}[section]

\newtheorem{lemma}[proposition]{Lemma}
\newtheorem{theorem}[proposition]{Theorem}

\DeclareMathOperator{\vol}{vol}
\DeclareMathOperator{\bdry}{\partial}
\DeclareMathOperator{\diam}{diam}
\DeclareMathOperator{\aplim}{ap \lim}
\DeclareMathOperator{\loc}{loc}
\DeclareMathOperator{\graph}{graph}

\DeclareMathOperator{\BV}{BV}
\DeclareMathOperator{\mass}{mass}
\DeclareMathOperator{\spt}{spt}
\DeclareMathOperator{\restrict}{\llcorner}

\newcommand{\norm}[1]{\left\| #1 \right\|}

\newcommand{\R}{\mathbb{R}}
\newcommand{\Rn}{\mathbb{R}^n}
\newcommand{\csubset}{\subset \subset}
\newcommand{\SMA}{\operatorname{SMA}}
\newcommand{\MA}{\operatorname{MA}}
\newcommand{\DC}{\operatorname{DC}}

\newcommand{\D}{\mathbb{D}}
\newcommand{\I}{\mathbb{I}}
\newcommand{\G}{\Gamma}

\newcommand{\eps}{\epsilon}
\newcommand{\e}{\epsilon}

\title[An extension of Alexandrov's theorem]{An extension of Alexandrov's theorem on second derivatives of convex functions}
\author{Joseph H.G. Fu}

\email{fu@math.uga.edu}
\date{\today}
\address{ Department of Mathematics, 
University of Georgia, 
Athens, GA 30602, USA}

\thanks{Partially supported
  by NSF grant DMS-1007580}

\begin{document}
\maketitle

\begin{abstract}If $f$ is a function of $n$ variables that is locally $L^1$ approximable by a sequence of smooth functions satisfying local $L^1$ bounds on the determinants of the minors of the Hessian, then $f$ admits a second order Taylor expansion almost everywhere. This extends a classical theorem of A.D. Alexandrov, covering the special case in which $f$ is locally convex.
\end{abstract}

\section{Introduction}  Let $U\subset \Rn$ be open. We say that $f:U\to \R$ is {\bf twice differentiable at $x$} if there is a quadratic polynomial $Q_x$ such that
\begin{equation}\label{2nd derivative}
\lim_{y\to x} \frac {f(y)- Q_x(y)}{|y-x|^{2}} = 0.
\end{equation}
 In the 1930s A.D. Alexandrov \cite{alexandrov} proved that if $f:U\to \R$ is a locally convex function on a domain $U\subset \R^n$ then $f$ is twice differentiable almost everywhere.
In the present article we extend this conclusion to a much larger class of functions. 

\begin{theorem}\label{main theorem} Let $U \subset \Rn$ be open, and $f_1, f_2,\dots \in C^2(U)$ converge in $L^1_{loc}(U)$ to a function
$f$. Suppose that the absolute integrals of all minors of the Hessians of the $f_k$ are uniformly locally bounded, i.e.
\begin{equation}\label{hess bound}
\int_K \left|\det \left(\frac{\partial^2 f_k}{\partial x_i\partial x_j}\right)_{i\in I, j\in J}\right| \le C(K), \quad k=1,2,\dots
\end{equation}
whenever $K \csubset U$ and $I,J\subset\{1,\dots,n\}$ have the same cardinality. Then
\begin{enumerate} \item \label{one}the $f_k$ are locally uniformly bounded,
\item\label{two}$f_k\to f$ pointwise a.e. in $U$,
\item\label{three} both $\limsup_{k\to \infty} f_k$ and $\liminf_{k\to \infty} f_k$ are twice differentiable at a.e. $x\in U$.
\end{enumerate}
\end{theorem}

Under the hypotheses of the theorem we will say that $f_1,f_2,\dots$ is a {\bf strong approximation} of $f$. Strongly approximable functions belong to the class of {\bf Monge-Amp\`ere functions} introduced by the present author in  \cite{fu89-I} and generalized by R. Jerrard in \cite{jerrard07a, jerrard07b}. They are distinguished by the existence of an integral current $\mathbb D( f)$ in the cotangent bundle $T^*U$, representing graph of its differential. The key point is that $\mathbb D(f)$ is determined uniquely (if it exists) by a short list of inevitable conditions.

{\bf Remarks.} 1. By the BV compactness theorem,  the local $L^1$ convergence implies local $W^{1,1}$ convergence of a subsequence.

2. If $n=1$ then $f$ is strongly approximable iff it is expressible as the difference of two convex functions. If $n = 2$ then any difference of convex functions is strongly approximable, but for $n\ge 2$ there exist strongly approximable functions that are not differences of convex functions. Furthermore there are pairs of strongly approximable functions whose sum is not strongly approximable (see below). If $n \ge 3$ then it is not known whether the a difference of convex functions is strongly approximable, or even whether it is Monge-Amp\`ere.
\vskip.05in

Alexandrov's theorem is a special case of Theorem \ref{main theorem}: if $f$ is convex then any locally uniform approximation by smooth convex functions (as may be obtained, for example, by convolution with an approximate identity) is a strong approximation. But Theorem \ref{main theorem} is significantly stronger: even though a convex function $f$ may fail to have first derivatives in the usual sense on a dense set, G. Alberti and L. Ambrosio \cite{alb-amb} observed  that such $f$ admits a multiple-valued differential everywhere, whose graph transforms into the graph of a Lipschitz function under the linear change of variable $(x,y) \mapsto (x+y , x-y)$ of $T^*\Rn\simeq \Rn\times \Rn$. From this point of view, Alexandrov's theorem appears as a consequence of Rademacher's theorem on the almost everywhere differentiability of Lipschitz functions. 

On the other hand, there exist strongly approximable functions displaying much wilder behavior. For example, they include the Sobolev space $W^{2,n}_{\loc}(U)$ of functions with second distributional derivatives in $L^n_{\loc}(U)$, since H\"older's inequality implies that convolution with an approximate identity yields a strong approximation.
 Meanwhile, Hutchinson-Meier \cite{HuMe} observed that if $n\ge 2$ then
\begin{equation}
f_{HM}(x_1,\dots,x_n) := x_1\sin \log \log |x|^{-1}
\end{equation}
belongs to $W^{2,n}(U)$ for sufficiently small neighborhoods $U$ of $0$, while $\nabla f_{HM}$ oscillates infinitely often between $(\pm 1,0,\dots,0) +o(x)$ as $x \to 0$  along the $x_1$ axis. 
This implies first of all that $f_{HM}$ cannot be expressed as a difference of two convex functions. Furthermore, by cutting off, translating, rotating and multiplying by suitable constants, it is easy to construct a $W^{2,n}(\Rn)$-convergent sum whose differential has a graph that is dense in $T^*\Rn$.  

For $n=2$ the function $g(x,y):= f_{HM}(x,y) + |y|$ is not Monge-Amp\`ere, and hence not strongly approximable. For the restriction of $g$ to $\R^2-\{0\}$ is Monge-Amp\`ere, and along the $x_1$ axis the fibers of the differential current $\mathbb D(g)$ include the line segments with endpoints $\nabla f_{HM} \pm (0,1)\in \R^2$. In view of the oscillation described above it follows that $\mathbb D(g)$ does not have finite mass above any neighborhood of the origin. This example is easily extended to higher dimensions.

 It is natural to conjecture that Theorem \ref{main theorem} applies to all Monge-Amp\`ere functions  --- indeed it seems plausible that every Monge-Amp\`ere function is strongly approximable. However, this is just one of many perplexing questions about Monge-Amp\`ere functions: for example, if $n > 2$ we do not even know whether such functions are necessarily continuous, or even locally bounded.

{\bf Acknowledgements.} I would like to thank Bob Jerrard for stimulating conversations on the topics discussed here, and also the Universit\`a di Trento for their hospitality during part of this work. I am extremely grateful also to the referee, who provided a much shorter and more direct proof of the main theorem. In fact the new proof yields a stronger statement than the original.

\section{Basic facts}\label{subsect:lemmas} 

\subsection{Some measure theory} Put  $B(x,r)\subset \Rn$ for the open ball of radius $r$ about $x$, and $\omega_n$ for the volume of $B(0,1)$. Recall that if $\mu$ is a Radon measure on $U \subset \Rn$ then its {\it density} at $x\in U$ is
$$
\Theta(\mu,x):= \lim_{r\downarrow 0} \frac {\mu(B(x,r))}{\omega_n r^n},
$$
provided the limit exists.
In fact the limit exists for a.e. $x \in U$ with respect to Lebesgue measure, and defines a Lebesgue-integrable function of $x$, whose integrals yield the absolutely continuous part of $\mu$ with respect to the Lebesgue decomposition into absolutely continuous and singular parts (cf. \cite{folland}, Thm. 3.22). In particular, if $\mu$ is singular with respect to the Lebesgue measure then
$\Theta(\mu,x) = 0$ for a.e. $x \in \Rn$.

\subsection{Absolute Hessian determinant measures} Suppose that $f_1,f_2,\dots \to f$ is a strong approximation. For $d=0,\dots,n$, and $k=1,2,\dots$, we define the measures $\nu_{k,d}$ on $U$ by
\begin{equation}\nu_{k,d}(S):=\sum_{I,J\subset \{1,\dots,n\}, |I|=|J|= d}\int_S  \left|\det \left[\frac{\partial^2 f_k}{\partial x^i
\partial x^j}\right]_{i \in I, j \in J} \right|.
\end{equation}
Taking subsequences, we may assume that each sequence $\nu_{k,d}, k = 1,2,\dots$, converges weakly to a Radon measure $\nu'_d,\ d =0,\dots,n$. We will refer to any Radon measure $\nu_d\ge \nu'_d$ as an {\bf absolute Hessian determinant measure of degree $d$} for the strong approximation $f_1,f_2,\dots \to f$. When the approximation is understood we will also refer to such $\nu_d$ as an  absolute Hessian determinant measure of degree $d$ for $f$.

For quadratic polynomials $Q:\R^n\to \R$, put $\norm Q$ to be the maximum of the absolute values of the coefficients. The following lemma is obvious. 

\begin{lemma}\label{quadratic perturbation} There are constants $ C_{n,d}$ with the following property. Let $f\in L^1_{loc}(U)$ be strongly approximable, and let $\nu_1,\dots,\nu_n$ be absolute Hessian determinant measures for $f$ arising from a strong approximation $f_1,f_2,\dots\to f$. Given any quadratic polynomial $Q$, the sequence $f_1+Q,f_2+Q,\dots$ is then a strong approximation of $f+Q$, and the measures
\begin{equation}
\tilde \nu_d := C_{n,d} \sum_{i+j = d} \norm Q^i \nu_j
\end{equation}
are absolute Hessian determinant measure of degrees $d=1,2,\dots,n$ for this approximation.
\end{lemma}
{\bf Remark.} Since the number of $d \times d $ minors of  an $n\times n$ matrix is $\binom n d^2$, and there are $\sum_{i=0}^d \binom d i^2 = \binom {2d} d$ subminors of each such minor, the rather extravagant value $C_{n,d} = d!\binom n d^2 \binom {2d} d$ works.

\subsection{An inequality from multivariable calculus} The key fact that makes the main theorem work is the following elementary classical inequality about $C^2$ functions.

\begin{lemma}\label{lem:alex} Let $U\subset \Rn$ be open, and $F\in C^2(U)$. If $V \subset\subset U$ is open then
\begin{equation} \int_V \left|\det D^2 F\right| \ge \omega_n \left(\frac{\sup_V | F| - \sup_{\bdry V}| F|}{\diam V} \right)^n
\end{equation}
\end{lemma}
\begin{proof} This is a weakened form of Lemma 9.2 of \cite{gt}.
%
\end{proof}

\section{An extension of a special case of a theorem of Calder\'on and Zygmund}\label{cz section}  

We say that $f:\R\supset U \to \R$ {\bf admits a $k$th derivative in the $L^1$ sense} at $x$ if there exists a polynomial $Q$ of degree $k$ such that
\begin{equation}\label{L1 2nd derivative}
r^{-n}{  \int}_{B(x,r)}\left| {f(y)- Q(y)} \right| \, dy= o(r^k)
\end{equation}
as $r\downarrow 0$. 


%

The next proposition generalizes a result of Calder\`on-Zygmund \cite{cz}.
The original result  (or rather the very special case of it that we have in mind) states that a function with distributional second derivatives in $L_{\loc}^1$ admits a second derivative in the $L^1$ sense a.e. By a straightforward adaptation of the argument of \cite{cz} we prove that this conclusion is true if the distributional second derivatives are only locally finite signed measures. Recall that the space $\BV_{\loc}(U)$ of functions of locally bounded variation consists of all locally integrable functions whose distributional gradients are (vector) measures (cf. \cite{giusti}). 

\begin{lemma} \label{diff lemma} If $g \in \BV_{\loc}(U)$ then $g$ is differentiable in the $L^1$ sense at a.e. $x\in U$.
\end{lemma}
\begin{proof} This is an immediate consequence of Thm. 6.1.1 of \cite{ev-gar}.
\end{proof}

\begin{proposition}\label{cz prop} If the distributional gradient of $f \in L^1(U)$ lies in $\BV_{\loc}(U)$, then $f $  admits a second derivative in the $L^1$ sense a.e. in $U$.
\end{proposition}
%
%

\begin{proof} By Lemma \ref{diff lemma} it is enough to show that there exists an $L^1$ quadratic Taylor approximation for $f$ at $0$ provided $\nabla f$ is differentiable in the $L^1$ sense at $0$. We may assume that $f(0) = \nabla f(0) = D^2f(0)=0$, where $D^2f(0)$ is the $L^1$ derivative of $\nabla f$ at $0$. Put 
$$
G(\rho):= \int_{B(0,\rho)} \frac{|\nabla f(x)|}{|x|^{n-1}}\, dx, \quad F(\rho):= \int_{B(0,\rho)} {|\nabla f(x)|}\, dx.
$$
Then $F,G$ are both absolutely continuous on $[0,\infty)$, with 
$G'(\rho)= \frac{F'(\rho)}{\rho^{n-1}},$
and $F(\rho) = o(\rho^{n+1})$. Integrating by parts, it follows that 
$G (\rho) = o(\rho^2)$ as $\rho\downarrow 0$. On the other hand,
\begin{align*}
\int_{B(0,\rho)} |f| &\le C\int_0^\rho r^{n-1}\, dr\, \int_{S^{n-1} } dv\, \int_0^r |Df(sv)|\,ds \\
&\le C\int_0^\rho r^{n-1}\, dr\, \int_{S^{n-1} } dv\, \int_0^\rho |Df(sv)|\,ds \\
&= C\rho^n  \int_{S^{n-1} } dv\, \int_0^\rho |Df(sv)|\,ds \\
&= C\rho^n G(\rho) = o(\rho^{n+2}),
\end{align*}
which gives the result.
\end{proof}

\section{Proof of Theorem \ref{main theorem}} For the rest of the paper we take as given the hypotheses of Theorem \ref{main theorem}.
\subsection{Proof of conclusion \eqref{one}}
The proof is a simplified version of the proof of the other conclusions.
\begin{proposition} The $f_k$ are locally uniformly bounded.
\end{proposition}
\begin{proof} 
 Given $x \in U$ and $s>0$, put 
\begin{equation}\label{def Q}
Q(x,s):= [x_1-s,x_1 + s]\times \dots \times [x_n-s,x_n + s]
\end{equation}
for the closed cube of side $2s$ centered at $x$. Let $r_0>0$ be small enough that $Q(x,r_0) \subset U$. We will show that
\begin{equation}\label{eq:sup}
\sup_{Q\left(x,\frac {r_0} 2\right)} |f_k| \le C
\end{equation}
for some constant $C$ independent of $k$.

For $z \in (0,r_0)^n$, put
\begin{equation}\label{def R}
R(z)= R(x,z):=[x_1-z_1,x_1+z_1]\times\dots\times[ x_n- z_n,x_n+z_n].
\end{equation}
Let $\mathcal F_d(z)$ denote the set of $d$-dimensional faces of $R(z)$. For $F \in \mathcal F_d(z)$, put $\vec F$ for the $d$-dimensional affine space it generates.

Consider the functions $g_{k,d}: (0,r_0)^n\to \R$ given by
\begin{align}
\notag g_{k,d}(z) &:= \sum_{ F\in \mathcal F_d(z)} \int_F  \left|\det D^2(\left. f_k\right|_{F})\right|\\
\label{g ineq}&\le \sum_{F\in \mathcal F_d(z)} \int_{\vec F\cap Q(x, r_0)}  \left|\det D^2(\left. f_k\right|_{\vec F})\right|, \quad d=1,\dots,n, \\
\notag g_{k,0}(z)&:= \sum_{y \in \mathcal F_0(z)} |f_k(y)|,\\
\notag G_k&:= \sum_{d=0}^n g_{k,d}, \quad k=1,2,\dots.
\end{align}
Then by Fubini's theorem and \eqref{g ineq}
\begin{align*}
\int_{(0,r_0)^n} g_{k,d}& \le r_0^d\sum_{|I|=d}\int _{Q(x,r_0)}\left|\det \left[ \frac{\partial^2f_k}{\partial x_i \partial x_j}\right]_{i,j\in I}\right|\le C, \quad d=1,\dots, n \\
\int_{(0,r_0)^n} g_{k,0}& = \int_{Q(x,r_0)} |f_k| \le C,
\end{align*}
for some constant $C$, independent of $k$, where the sum is over all subsets $I\subset \{1,\dots,n\}$ of the indicated cardinality.

By Fatou's lemma,
\begin{align*}
\int_{(0,r_0)^n} \liminf_{k\to \infty }G_k &\le \liminf_{k\to \infty }\int_{(0,r_0)^n} G_k \le (n+1)C.
\end{align*}
Taking a subsequence if necessary we may therefore find $z^* \in (0,r_0)^n -(0,\frac{r_0}2)^n$ such that $g_{k,d}(z^*)\le G_k(z^*) \le \frac{2^n}{2^n-1}(n+1)C  r_0^{-n}=:C'$ for all sufficiently large $k$.

In particular
$$
|f_k(v)| < C'
$$
for all of the vertices $v \in \mathcal F_0(z^*)$ of the rectangle $R(z^*)$.
Applying Lemma \ref{lem:alex} to the faces of $R(z^*)$ we find that
$$
\sup_{F\in \mathcal F_{d}(z^*)} |f_k| \le \sup_{F\in \mathcal F_{d-1}(z^*)} |f_k| + \omega_{d}^{-\frac 1 d}g_{k,d}(z^*)^{\frac 1 d} d^{\frac 1 2} r_0\le \sup_{F\in \mathcal F_{d-1}(z^*)} |f_k| + \omega_{d}^{-\frac 1 d}(C')^{\frac 1 d}  d^{\frac 1 2} r_0,
$$
$d=1,\dots,n$.
Proceeding by induction on $d$, the case $d=n$ yields \eqref{eq:sup}.
\end{proof}

\subsection{Proof of conclusions \eqref{two} and \eqref{three}} The proof of Theorem \ref{main theorem} will be completed in the following Proposition. Put $\nu:= \nu_1+\dots+\nu_n$ for the sum of the absolute Hessian determinant measures of $f$, and $\nu = \nu_{ac}+ \nu_s$ the decomposition of $\nu$ into its absolutely continuous and singular parts. Let $\phi$ denote the density function of $\nu_{ac}$ with respect to Lebesgue measure. We put also
$$
\bar f:= \limsup_k f_k, \quad \underline f:= \liminf_k f_k
$$

\begin{proposition} Suppose $x_0\in U$, and that
\begin{enumerate}
\item $\Theta (\nu_s,x_0) =0$;
\item $x_0$ is a Lebesgue point of $f$, of $\nabla f$ and of $\phi$;
\item $f$ is twice differentiable in the the $L^1$ sense at $x_0$.
\end{enumerate}
 Then $\bar f(x_0)=\underline f(x_0)$, and both $\bar f,\underline f$ are twice differentiable at $x_0$.
\end{proposition}
\begin{proof} We may assume that $x_0=0$, and by Lemma \ref{quadratic perturbation} we may also assume that $f(0) = \nabla f(0) = D^2f(0) = 0$, where $D^2f(0)$ is the $L^1$ second derivative of $f$ at $0$. 
Let $\eps \in(0,1)$ be given, and take $r_0>0$ small enough that if $r\in (0,r_0)$ then, referring to the definition \eqref{def Q},
\begin{equation}\label{eq:small enough}
\nu_s(Q(0,2r)) + \int_{Q(0,2r)} \left(\frac {|f|}{\eps^2 r^2}+ |\phi(x)-\phi(0)| \right)  < (\eps r)^n.
\end{equation}

Let $0\ne x \in Q(0,r_0)$. We will show that for large $k$
\begin{equation}\label{eq:star}
|f_k(x)| < M\eps^2 |x|^2
\end{equation}
 where $M$ depends only $\phi(0)$ and the dimension $n$. 

Put $r:= \max(|x_1|,\dots,|x_n|)$ and 
\begin{equation}\label{eq:def g}
g_{k,d}:= \sum_{|I|=|J|= d}\left|\det\left[\frac{\partial^2 f_k}{\partial x_i\partial x_j}\right]_{i\in I, j\in J}\right|.
\end{equation}
Then
\begin{align}
\notag\limsup_{k\to \infty} \int_{Q(x,\eps r)} \sum_d g_{k,d}&\le \nu(\bar Q(x,\eps r))\\
\notag& = \nu_s(\bar Q(x,\eps r)) + \int_{Q(x,\eps r)} \phi \\
\label{eq:small int}&\le \nu_s(Q(0,2 r)) + (2\eps r)^n \phi(0)+ \int_{Q(0,2 r)} |\phi(y) -\phi(0)|  \\
\notag&< (1+2^n\phi(0)) (\eps r)^n
\end{align}
by \eqref{eq:small enough}.

Referring to the definitions of \eqref{def R} ff., consider the functions $h_k: (0,\eps r)^n\to \R$ given by
\begin{align*}
h_{k,0}(z)&:= \frac 1 {\eps^2 r^2}\sum_{w \in \mathcal F_0(z)} |f_k(w)|, \\
h_{k,d}(z) &:= \sum_{ F\in \mathcal F_d(z)} \int_F  \left|\det D^2(\left. f_k\right|_{F})\right|\\
&\le \sum_{F\in \mathcal F_d(z)} \int_{\vec F\cap Q(x,\eps r)}  \left|\det D^2(\left. f_k\right|_{\vec F})\right|, \quad d=1,\dots,n.
\end{align*}
Clearly
\begin{equation*}
\int_{(0,\eps r)^n} h_{k,0}(z)= \frac 1 {\eps^2r^2}\int_{Q(x,\eps r)}| f_k(w)|<(\eps r)^n
\end{equation*}
for large values of $k$, by \eqref{eq:small enough} and the local $L^1$ convergence $f_k\to f$.
Furthermore
\begin{align*}
\sum_{d=1}^n (\eps r)^{-d}\int_{(0,\eps r)^n} h_{k,d}(z) \, dz&\le\sum_{d=1}^n (\eps r)^{-d}\int_{(0,\eps r)^n} dz\sum_{F\in \mathcal F_d(z)} \int_{\vec F\cap Q(x,\eps r)}  \left|\det D^2(\left. f_k\right|_{\vec F})\right| \\
&= \sum_{d=1}^n \int_{Q(x,\eps r)} g_{k,d}\\
&< (\eps r)^{n} (1+2^n\phi(0)), \quad d=1,\dots,n,
\end{align*}
for $k$ large enough, by \eqref{eq:small int} and Fubini's theorem. Putting
$$
H_k(z):= h_{k,0}(z) + \sum_{d=1}^n (\eps r)^{-d} h_{k,d}(z),
$$
Fatou's lemma yields
\begin{align*}
\int_{(0,\eps r)^n}\liminf_{k\to \infty} H_k(z)\, dz &\le\liminf_{k\to \infty} \int_{(0,\eps r)^n}H_k(z)\, dz \\
&< (\eps r)^n (2+2^n\phi(0)).
\end{align*}
Therefore, taking subsequences as necessary, there exists $z^*\in (0,\eps)^n$  such that
\begin{equation}\label{eq:Hk}
 H_k(z^*) < 2+2^n\phi(0)=:C
\end{equation}
 for all sufficiently large $k$.

For $k$ large the values of $f$ at the vertices $v\in \mathcal F_0(z^*)$ of $R(z^*)$ satisfy
$$
|f_k(v)| <C(\eps r)^2.
$$
Proceeding by induction on the dimension $d$, we claim that for large $k$
$$
\sup_{y \in F\in \mathcal F_d(z^*)} |f_k(y)| \le C A_d \eps^2 r^2, \quad d=1,\dots, n,
$$
where $C$ is the constant from \eqref{eq:Hk} and $A_d$ depends only on $d$.
To see this we observe that \eqref{eq:Hk} yields
\begin{equation}\label{eq:face integrals}
\int_{F} \left|\det D^2\left.(f_k\right|_{F})\right|  <C (\eps r)^d, \quad d=1,\dots n.
\end{equation}
for $F \in \mathcal F^d(z^*)$ and large $k$.
Thus by Lemma \ref{lem:alex}, for each such face $F$
\begin{align*}\label{eq:apply alex}
(\sup_{F} |f_k| - \sup_{\partial F} |f_k|)^d&\le \left({\diam F}\right)^d\omega_d^{-1}  C(\eps r)^d \\
&\le C\omega_d^{-1} (\eps r)^{2d}{d}^{\frac  d 2}
\end{align*}
whence
\begin{align*}
\sup_{F} |f_k| &\le  \sup_{\partial F} |f_k|+C^{\frac 1 d}\omega_d^{-\frac 1 d} \sqrt d (\eps r)^2\\
& \le( CA_{d-1} +C^{\frac 1 d}\omega_d^{-\frac 1 d} \sqrt d )(\eps r)^2,
\end{align*}
as claimed.

Now \eqref{eq:star} holds with $M= CA_n$. A similar, easier argument shows that $\lim_k f_k(0) = 0 = f(0)$.
\end{proof}

\end{document}